\documentclass[arxiv]{agn_article}

\usepackage[tikzexternalize=false]{agn_all}
\usepackage{graphicx}
\usepackage{amsmath,mathtools,amssymb}
\usepackage{agn_macros}
\usepackage{float}
\usepackage{hyperref}

\newcommand{\citet}[2][]{\citeauthor{#2} \cite[#1]{#2}}
\newcommand{\Wiso}{W_{\mathrm{iso}}}
\newcommand{\Wvol}{W_{\mathrm{vol}}}

\DeclareMathOperator{\CSO}{CSO}
\newcommand{\CSOn}{\CSO(n)}
\pgfplotsset{
every axis/.append style={width=0.6\textwidth, height=6cm},
every axis plot/.append style={no markers, thick},
label style={font=\small},
tick label style={font=\small},
legend pos={outer north east},
cycle list={red, green, cyan, yellow}
}
	\newcommand\OmegaSetDefaults{
		\OmegaSetGridSize{.1}{.1}
		\OmegaSetOutlineSampleCount{25}
		\OmegaSetGridSampleCount{5}
		\OmegaSetOutlineToUnitCircle
		\OmegaSetDeformationToId
		\OmegaSetOutlineStyle{thick, color=black}
		\OmegaSetShadingStyle{left color = lightgray, right color = black, opacity = .3}
		\OmegaSetGridStyle{help lines}
		\OmegaSetSmoothOutline
	}
	\newcommand{\OmegaSetOutlineStyle}[1]{
		\tikzset{outlinestyle/.style={#1}}
	}
	\newcommand{\OmegaSetShadingStyle}[1]{
		\tikzset{shadingstyle/.style={#1}}
	}
	\newcommand{\OmegaSetGridStyle}[1]{
		\tikzset{gridstyle/.style={#1}}
	}
	
	\newcommand{\OmegaSetSmoothOutline}{
		\tikzset{outlineplotstyle/.style={smooth}}
	}
	\newcommand{\OmegaSetOutline}[3]{
		\pgfmathdeclarefunction*{omegagammax}{1}{\pgfmathparse{#1(##1)}}
		\pgfmathdeclarefunction*{omegagammay}{1}{\pgfmathparse{#2(##1)}}
		\pgfmathsetmacro{\OmegaParameterlength}{#3}
	}
	\newcommand{\OmegaSetDeformation}[2]{
		\pgfmathdeclarefunction*{omegaphix}{2}{\pgfmathparse{#1(##1, ##2)}}
		\pgfmathdeclarefunction*{omegaphiy}{2}{\pgfmathparse{#2(##1, ##2)}}
	}
	\newcommand{\OmegaSetGridSize}[2]{
		\pgfmathsetmacro{\OmegaGridx}{#1}
		\pgfmathsetmacro{\OmegaGridy}{#2}
	}
	\newcommand{\OmegaSetOutlineSampleCount}[1]{
		\pgfmathsetmacro{\OmegaOutlineSampleCount}{#1}
	}
	\newcommand{\OmegaSetGridSampleCount}[1]{
		\pgfmathsetmacro{\OmegaGridSampleCount}{#1}
	}
	\newcommand{\OmegaSetOutlineToCircle}[3]{
		\pgfmathdeclarefunction*{circx}{1}{\pgfmathparse{#1 + #3*cos(##1)}}
		\pgfmathdeclarefunction*{circy}{1}{\pgfmathparse{#2 + #3*sin(##1)}}
		\OmegaSetOutline{circx}{circy}{360}
	}
	\newcommand{\OmegaSetOutlineToUnitCircle}{
		\OmegaSetOutlineToCircle{0}{0}{1}
	}

	\newcommand{\OmegaSetDeformationToId}{
		\pgfmathdeclarefunction*{idx}{2}{\pgfmathparse{##1}}
		\pgfmathdeclarefunction*{idy}{2}{\pgfmathparse{##2}}
		\OmegaSetDeformation{idx}{idy}
	}
	
	\newcommand{\OmegaDraw}{
		\OmegaDrawCustomDeformation{omegaphix}{omegaphiy}
	}
	\newcommand{\OmegaDrawCustomDeformation}[2]{
		\OmegaDrawFull{\OmegaGridx}{\OmegaGridy}{\OmegaParameterlength}{omegagammax}{omegagammay}{#1}{#2}{\OmegaOutlineSampleCount}{\OmegaGridSampleCount}
	}
	\newcommand{\OmegaDrawFull}[9]{	
		\begin{scope}
			\def\conversionFromBoundingBoxUnit{0.03514598035}
			\tikzset{borderdomainstyle/.style={domain=0:#3, variable=\t, samples=#8}}
			\tikzset{griddomainstyle/.style={smooth, variable=\t, samples=#9}}
			\newdimen\leftBoundWrongUnit
			\newdimen\rightBoundWrongUnit
			\newdimen\lowerBoundWrongUnit
			\newdimen\upperBoundWrongUnit

			\shade[shadingstyle]
				plot[borderdomainstyle, outlineplotstyle] ({#6(#4(\t), #5(\t))}, {#7(#4(\t), #5(\t))});

			\begin{scope}
				\coordinate (oldLowerLeftBound) at (current bounding box.south west);
				\coordinate (oldUpperRightBound) at (current bounding box.north east);

				\pgfresetboundingbox
				\path[use as bounding box]	
					plot[borderdomainstyle, outlineplotstyle] ({#4(\t)}, {#5(\t))});
				\coordinate (lowerLeftBound) at (current bounding box.south west);
				\coordinate (upperRightBound) at (current bounding box.north east);
			\end{scope}
			\pgfextractx{\leftBoundWrongUnit}{\pgfpointanchor{lowerLeftBound}{center}}
			\pgfextractx{\rightBoundWrongUnit}{\pgfpointanchor{upperRightBound}{center}}
			\pgfextracty{\lowerBoundWrongUnit}{\pgfpointanchor{lowerLeftBound}{center}}
			\pgfextracty{\upperBoundWrongUnit}{\pgfpointanchor{upperRightBound}{center}}
			\pgfmathsetmacro{\leftBound}{\conversionFromBoundingBoxUnit*\leftBoundWrongUnit}
			\pgfmathsetmacro{\rightBound}{\conversionFromBoundingBoxUnit*\rightBoundWrongUnit}
			\pgfmathsetmacro{\lowerBound}{\conversionFromBoundingBoxUnit*\lowerBoundWrongUnit}
			\pgfmathsetmacro{\upperBound}{\conversionFromBoundingBoxUnit*\upperBoundWrongUnit}

			\begin{scope}
				\clip
					plot[borderdomainstyle, outlineplotstyle] ({#6(#4(\t), #5(\t))}, {#7(#4(\t), #5(\t))});

				\pgfmathsetmacro{\horizontalGridCount}{(\upperBound-\lowerBound)/#2}
				\foreach \b in {1,2,...,{\horizontalGridCount}}{
					\draw[gridstyle]
						plot[griddomainstyle, domain=\leftBound:\rightBound] ({#6(\t,{\lowerBound+(#2*\b)})},{#7(\t,{\lowerBound+(#2*\b})});
				}

				\pgfmathsetmacro{\verticalGridCount}{(\rightBound-\leftBound)/#1}
				\foreach \b in {1,2,...,{\verticalGridCount}}{
					\draw[gridstyle]
						plot[griddomainstyle, domain=\lowerBound:\upperBound] ({#6({\leftBound+(#1*\b)},\t)},{#7({\leftBound+(#1*\b)},\t)});
				}
			\end{scope}

			\draw[outlinestyle]
					plot[borderdomainstyle, outlineplotstyle] ({#6(#4(\t), #5(\t))}, {#7(#4(\t), #5(\t))});

			\path[use as bounding box] (oldLowerLeftBound) rectangle (oldUpperRightBound);
		\end{scope}
	}

\begin{document}
\title{\vspace*{-2.75em}A note on non-homogeneous deformations with homogeneous Cauchy stress for a strictly rank-one convex energy in isotropic hyperelasticity}
\date{\today}
\author{%
	Eva Schweickert\thanks{%
		Corresponding author: Eva Schweickert,\quad Lehrstuhl f\"{u}r Nichtlineare Analysis und Modellierung, Fakult\"{a}t f\"{u}r Mathematik, Universit\"{a}t Duisburg-Essen, Thea-Leymann Str. 9, 45127 Essen, Germany; email: eva.schweickert@stud.uni-due.de%
	}, \quad%
	L.\ Angela Mihai\thanks{%
		L.\ Angela Mihai, \quad School of Mathematics, Cardiff University, Senghennydd Road, Cardiff, CF24 4AG, UK; email: MihaiLA@cardiff.ac.uk%
	}, \quad%
	Robert J.\ Martin\thanks{%
		Robert J.\ Martin,\quad Lehrstuhl f\"{u}r Nichtlineare Analysis und Modellierung, Fakult\"{a}t f\"{u}r Mathematik, Universit\"{a}t Duisburg-Essen, Thea-Leymann Str. 9, 45127 Essen, Germany; email: robert.martin@uni-due.de%
	} 
	\quad and\quad
	Patrizio Neff\thanks{%
		Patrizio Neff,\quad Head of Lehrstuhl f\"{u}r Nichtlineare Analysis und Modellierung, Fakult\"{a}t f\"{u}r	Mathematik, Universit\"{a}t Duisburg-Essen, Thea-Leymann Str. 9, 45127 Essen, Germany, email: patrizio.neff@uni-due.de%
		}%
}
\maketitle
\vspace*{-2.75em}
\begin{abstract}
	It has recently been shown that for a Cauchy stress response induced by a strictly rank-one convex hyperelastic energy potential, a homogeneous Cauchy stress tensor field cannot correspond to a non-homogeneous deformation if the deformation gradient has discrete values, i.e.\ if the deformation is piecewise affine linear and satisfies the Hadamard jump condition. In this note, we expand upon these results and show that they do not hold for arbitrary deformations by explicitly giving an example of a strictly rank-one convex energy and a non-homogeneous deformation such that the induced Cauchy stress tensor is constant.
	In the planar case, our example is related to another previous result concerning criteria for generalized convexity properties of conformally invariant energy functions, which we extend to the case of strict rank-one convexity. 
\end{abstract}

\\[.65em]

\vspace{-1.5em}
{\parskip=-0.5mm \tableofcontents}

\section{Introduction}
\label{section:introduction}
In recent contributions \cite{agn_mihai2018hyperelastic,agn_mihai2016hyperelastic} we have exhibited both planar and three-dimensional inhomogeneous (i.e.\ non-affine) deformations for which the associated Cauchy stress tensor is homogeneous (i.e.\ constant). More specifically, we considered configurations with discrete deformation gradients, i.e.\ piecewise affine linear (\enquote{laminate}) deformations satisfying the Hadamard jump condition. The hyperelastic energy potentials which induced a homogeneous Cauchy stress for these deformations happened to be non-elliptic, i.e.\ not rank-one convex on the group $\GLpn$ of invertible matrices with positive determinant. Moreover, we showed that this loss of ellipticity is essential for the observed phenomenon to occur by proving that for a strictly rank-one convex energy function, the Cauchy stress corresponding to a non-trivial laminate (i.e.\ a non-affine but piecewise affine deformation) is never homogeneous \cite{agn_neff2016injectivity}. Here, a function $W\col\GLpn\to\R$ is called strictly rank-one convex if
\[
	W(F+t(\xi\otimes\eta)) < (1-t)\.W(F) + t\.W(F+\xi\otimes\eta)
	\qquad\text{for all }\;
	t\in[0,1]
\]
for all $F\in\GLpn$ and all $\xi,\eta\in\R^n$ such that $F+\xi\otimes\eta\in\GLpn$.

This observation raises the question whether our result is restricted to the laminate case or whether the Cauchy stress induced by a strictly rank-one convex function is non-homogeneous for \emph{any} non-affine deformation.\footnote{Note that the brief summary of our result from \cite{agn_mihai2016hyperelastic} in the abstract of \cite{agn_neff2016injectivity} could easily be misread to claim the latter.} In the following, we will negatively answer this further-reaching conjecture by giving an explicit example (both for the two-dimensional and the three-dimensional case) of a strictly rank-one convex energy and a non-affine deformation such that the corresponding Cauchy stress tensor is constant. In particular, this result also implies that strict rank-one convexity is, in general, not connected to the invertibility of the Cauchy stress-stretch relation.

The energy functions in our examples are not highly pathological in nature; in fact, they are not only objective and isotropic, but also satisfy common constitutive conditions such as the correct growth behaviour for singular deformation gradients, coercivity and the uniqueness of a stress-free reference configuration, cf.\ Remark \ref{remark:constitutivePropertiesOfExample}. Furthermore, the non-affine deformations we consider are \emph{conformal}, i.e.\ locally angle preserving, cf.\ Section \ref{sec:conformalmappings}.

Our approach is structured as follows. After a brief introduction of conformal mappings, we consider strictly rank-one convex examples of so-called \emph{conformally invariant} energ functions and discuss their connection to conformal mappings, both in the finite and the linearized case.
By additively coupling these purely isochoric energies with an appropriate volumetric expression, we then construct energy functions which are physically viable, but still induce a constant Cauchy stress tensor field for certain conformal deformation mappings, which proves our main result as given in Proposition \ref{prop:final}. We will also point out that our result does not contradict the statement made in \cite{agn_neff2016injectivity} by showing that the (conformal) mappings we use in our examples do not satisfy the essential Hadamard jump condition.

\section{Conformal mappings}\label{sec:conformalmappings}
As indicated in Section \ref{section:introduction}, our examples of non-affine deformations corresponding to constant Cauchy stress are \emph{conformal mappings}, i.e.\ of the form
\[
	\varphi\col\Omega\to\R^n
	\qquad\text{with}\quad
	\grad\varphi(x)\in\CSOn
	\quad\text{for all }\;x\in\Omega
	\,,
\]
where $\Omega\subset\R^n$ is the elastic body in its reference configuration and
\[
	\CSOn \colonequals \Rp\cdot\SOn = \{ \lambda\cdot R \,\setvert\, \lambda>0\,,\; R\in\SOn \}
\]
denotes the \emph{special conformal group}; here, $\Rp=(0,\infty)$ is the set of positive real numbers and $\SOn$ denotes the special orthogonal group. In particular, a mapping $\varphi$ is conformal if and only if for each $x\in\Omega$, there exist $\lambda(x)\in\Rp$ and $R(x)\in\SOn$ such that
\[
	\grad\varphi(x)=\lambda(x)\cdot R(x)
	\qquad\text{or, equivalently,}\qquad
	\frac{\grad\varphi(x)^T\grad\varphi(x)}{\det(\grad\varphi(x)^T\grad\varphi(x))^{\afrac1n}} = \id
	\,.
\]
In the two-dimensional case, the (planar) conformal mappings on $\Omega\subset\R^2$ can be identified with the holomorphic functions $g\colon\Omega\subset\C\to\C$ with $g'(z)\neq0$ for all $z\in\Omega$.

We consider an example of a non-affine conformal mapping on some open set $\Omega\subset\R^2$, as shown in Fig.\ \ref{fig:mobius}.
For $\Omega\subset\R^2\setminus\{0\}$, let
\[
	\varphi\colon\Omega\to\R^2\,,\quad \varphi(x) = \varphi\matr{x_1\\x_2} = \frac{1}{\norm{x}^2}\cdot\matr{x_1\\-x_2}\,.
\]
Then
\begin{align}\label{eq:detconformal}
	\grad\varphi(x)=\left(\begin{array}{cc}
									\frac{\norm{x}^2-2x_1^2}{\norm{x}^4} & \frac{-2x_1x_2}{\norm{x}^4}\\
									\frac{2x_1x_2}{\norm{x}^4} & \frac{-\norm{x}^2+2x_2^2}{\norm{x}^4}	
									\end{array}\right)=
									\frac{1}{\norm{x}^2}\left(\begin{array}{cc}
															1-\frac{2x_1^2}{\norm{x}^2} & -\frac{2x_1x_2}{\norm{x}^2}\\
															\frac{2x_1x_2}{\norm{x}^2} & -1+\frac{2x_2^2}{\norm{x}^2}
															\end{array}\right)\,,
\end{align}
thus
\begin{align*}
	\left(\grad\varphi(x)\right)^T\grad\varphi(x)&= \frac{1}{\norm{x}^4}\cdot\left(\begin{array}{cc}
																					1-\frac{4x_1^2}{\norm{x}^2}+\frac{4x_1^4}{\norm{x}^4}+\frac{4x_1^2x_2^2}{\norm{x}^4} & \frac{2x_1x_2}{\norm{x}^2}-\frac{4x_1^3x_2}{\norm{x}^4}+\frac{2x_1x_2}{\norm{x}^2}-\frac{4x_1x_2^3}{\norm{x}^4}\\
																					\frac{2x_1x_2}{\norm{x}^2}-\frac{4x_1^3x_2}{\norm{x}^4}+\frac{2x_1x_2}{\norm{x}^2}-\frac{4x_1x_2^3}{\norm{x}^4} & \frac{4x_2^4}{\norm{x}^4}-\frac{4x_2^2}{\norm{x}^2}+1+\frac{4x_1^2x_2^2}{\norm{x}^4}	
																					\end{array}\right)
	=\frac{1}{\norm{x}^4}\cdot\id
\end{align*}
and $\det(\grad\varphi(x))=\frac{1}{\norm{x}^4}>0$.
Consequently, $\varphi$ is a non-affine conformal mapping with a decomposition $\grad\varphi(x)=\lambda(x)\cdot R(x)$ such that $\lambda(x)>0$ and $R(x)\in\SO(2)$ for all $x\in\Omega$; specifically,
\[
	\lambda(x) = \det(\grad\varphi(x))^{\afrac12} = \frac{1}{\norm{x}^2}
	\qquad\text{and}\qquad
	R(x) = \frac{\grad\varphi(x)}{\det(\grad\varphi(x))^{\afrac12}}
	= \matr{
		1-\frac{2x_1^2}{\norm{x}^2} & -\frac{2x_1x_2}{\norm{x}^2}\\
		\frac{2x_1x_2}{\norm{x}^2} & -1+\frac{2x_2^2}{\norm{x}^2}
	}
	\in\SO(2)\,.
\]
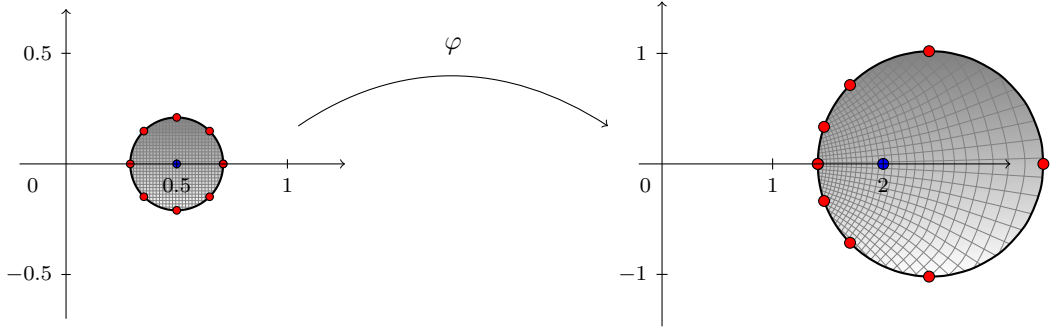
\begin{figure}[H]
	\centering
	\begin{tikzpicture}[scale=1.47,even odd rule]
	\OmegaSetDefaults
	\OmegaSetGridSize{.0147}{.0147}
	\OmegaSetOutlineSampleCount{42}
	\OmegaSetOutlineStyle{thick, color=black}
	\OmegaSetShadingStyle{color = lightgray}
	\OmegaSetGridStyle{help lines}
	\def\circleCenterX{0.5}
	\def\circleCenterY{0}
	\def\circleRadius{0.21}
	\def\stepSize{45}
	\OmegaSetOutlineToCircle{\circleCenterX}{\circleCenterY}{\circleRadius}
	\OmegaSetSmoothOutline
		\begin{scope}[scale=2.002]
			\OmegaSetGridSampleCount{7}
			\OmegaSetDeformationToId
			\OmegaDraw
			\OmegaSetGridSampleCount{21}
			\foreach \n in {0,\stepSize,...,360}{
				\draw[fill=red]({\circleCenterX+\circleRadius*cos(\n)},{\circleCenterY+\circleRadius*sin(\n)})circle(0.49pt);
			}
			\draw[fill=blue](\circleCenterX,\circleCenterY)circle(0.49pt);
			\draw[->] (0,-0.7) -- (0,0.7);
			\draw[->] (-0.21,0) -- (1.26,0);
			\draw (0.5,0.021) -- (0.5,-0.021) node[below] {\footnotesize $0.5$};
			\draw (0.021,0.5) -- (-0.021,0.5) node[left] {\footnotesize $0.5$};
			\draw (0.021,-0.5) -- (-0.021,-0.5) node[left] {\footnotesize $-0.5$};
			\draw (1,0.021) -- (1,-0.021) node[below] {\footnotesize $1$};
			\node[below] at (-0.15,-0.021){\footnotesize $0$};
		\end{scope}
		\begin{scope}[xshift=5.39cm,yshift=0cm]
		\def\moeConst{0}
		\def\moeScale{1}
			\pgfmathdeclarefunction*{moebiusx}{2}{%
				\pgfmathparse{\moeScale*(#1+\moeConst)/((#1+\moeConst)^2+(#2+\moeConst)^2)}%
			}
			\pgfmathdeclarefunction*{moebiusy}{2}{%
				\pgfmathparse{-\moeScale*(#2+\moeConst)/((#1+\moeConst)^2+(#2+\moeConst)^2)}%
			}
			\OmegaSetDeformation{moebiusx}{moebiusy}
			\OmegaDraw
			\foreach \n in {0,\stepSize,...,360}{
				\draw[fill=red]({moebiusx(\circleCenterX+\circleRadius*cos(\n),\circleCenterY+\circleRadius*sin(\n))},{moebiusy(\circleCenterX+\circleRadius*cos(\n),\circleCenterY+\circleRadius*sin(\n))})circle(1.4pt);
			}
			\draw[fill=blue]({moebiusx(\circleCenterX,\circleCenterY)},{moebiusy(\circleCenterX,\circleCenterY)})circle(1.4pt);
			\draw[->] (0,-1.47) -- (0,1.47);
			\draw[->] (-0.21,0) -- (3.15,0);
			\draw (1,0.05) -- (1,-0.05) node[below] {\footnotesize $1$};
			\draw (0.05,1) -- (-0.05,1) node[left] {\footnotesize $1$};
			\draw (0.05,-1) -- (-0.05,-1) node[left] {\footnotesize $-1$};
			\draw (2,0.05) -- (2,-0.05) node[below] {\footnotesize $2$};
			\node[below] at (-0.15,-0.05){\footnotesize $0$};
		\end{scope}
		
		\draw[->] (2.1,0.343) to[out=34.3,in=147] (4.9,0.343);
		\draw (3.5,.91) node[above]{$\varphi$};
	\end{tikzpicture}
	\caption{Visualization of the conformal mapping $\varphi\colon\Omega\subset\R^2\setminus\{0\}\to\R^2$ with $\varphi(x)=\frac{1}{\norm{x}^2}\cdot\binom{x_1}{-x_2}$, 
	showing that infinitesimal squares are rotated and scaled.}
	\label{fig:mobius}
\end{figure}
Note that for arbitrary dimension $n\geq2$ and $\widetilde{\Omega}\subset\R^n\setminus\{0\}$, the mapping $\widetilde{\varphi}\colon\widetilde{\Omega}\to\R^n$ with $\widetilde{\varphi}(x)=\frac{1}{\norm{x}^2}\left(x_1,-x_2,x_3,x_4, \dotsc,x_n\right)^T$
is conformal as well.
\subsection{Möbius transformations}
The above example is a special case of a so-called \emph{Möbius transformation}, i.e.\ a mapping $\varphi\col\Omega\subset\R^n\to\R^n$ given as the composition of finitely many reflections at hyperplanes and/or spheres.%
\footnote{%
	The reflection $s\col\R^n\setminus\{x_0\}\to\R^n$ at a sphere $S_{x_0}(r)=\{x\in\R^n\setvert\norm{x-x_0}^2=r^2\}$ with center $x_0\in\R^n$ and radius $r>0$ is defined by $s(x)=x_0+\frac{r^2}{\norm{x-x_0}^2}\cdot(x-x_0)$. Note that a Möbius transformation is orientation preserving if and only if the number of composed reflections is even.
}
\begin{remark}\label{rem:moebiusconformal}
	For $n\geq3$, every non-trivial conformal mapping is a Möbius transformation \cite{blair2000inversion}, while in the planar case $n=2$, the orientation-preserving Möbius transformations correspond to the complex functions of the form
	\begin{align}
		\ftilde\colon\C\to\C\,,\quad
		\ftilde(z)\colonequals\frac{az+b}{cz+d}\qquad\text{with }\;\; a,b,c,d\in\C\,,\; ad-bc\neq0 \,.
	\end{align}
\end{remark}
\section{Conformally invariant energy functions and strict rank-one convexity}\label{sec:convexenergiesconformalmappings}
A function $W\col\GLpn\to\R$ is called \emph{conformally invariant} if it is \emph{objective}, \emph{isotropic} and \emph{isochoric}, i.e.\ if
\[
	W(a\,Q_1FQ_2) = W(F)
	\qquad\text{for all }\;\;
	F\in\GLpn\,,\;
	a>0\,,\;
	Q_1,Q_2\in\SOn
	\,.
\]
In nonlinear elasticity, conformally invariant energy functions are generally not directly suited for modeling the elastic behaviour of a material due to their invariance under purely volumetric scaling. However, they are commonly coupled with a volumetric energy term of the form $F\mapsto f(\det(F))$ for some function $f\col(0,\infty)\to\R$. Energy functions of this type, also known as an additive \emph{volumetric-isochoric split} \cite{richter1948,agn_graban2019richter}, will be used in Section \ref{sec:results} in order to establish our main results.

\subsection{Conformally invariant energies in the planar case}

It is well known \cite{agn_martin2015rank} that in the planar case $n=2$, an energy function $W\col\GLpn\to\R$ is conformally invariant if and only if it can be expressed in the form
\begin{equation}\label{eq:planarConformallyInvariantRepresentations}
	W(F) = \psi(\K(F)) = \hhat(K(F))
	\qquad\text{with }\quad
	\psi,\hhat\col[1,\infty)\to\R
\end{equation}
for all $F\in\GLp(2)$, where
\[
	\K\col\GLp(2)\to[1,\infty)\,,
	\quad
	\K(F) = \frac12\.\frac{\norm{F}^2}{\det F}
	\qquad\text{and}\qquad
	K\col\GLp(2)\to[1,\infty)\,,
	\quad
	K(F) = \frac{\opnorm{F}^2}{\det F}
\]
denote the \emph{distortion}\footnote{Note that $\K(F)\geq1$ for all $F\in\GLp(2)$ due to the Hadamard inequality, with $\K(F)=1$ if and only if $F$ is conformal.} and the \emph{linear distortion} function, respectively; here, $\norm{X}=(\sum_{i,j=1}^n X_{ij}^2)^{\afrac12}$ is the Frobenius matrix norm and $\opnorm{X}$ is the operator norm, i.e.\ the largest singular value, of $X\in\Rnn$.

In the following, let $W\col\GLp(2)\to\R$ be of the form \eqref{eq:planarConformallyInvariantRepresentations} such that the representation $\psi$ of $W$ in terms of the classical distortion $\K$ is sufficiently smooth with $\psi^\prime(1)>0$ as well as strictly monotone increasing and convex on $[1,\infty)$. Since the mapping $\K$ itself is polyconvex \cite{agn_martin2015rank}, the function $W$ is polyconvex and thus rank-one convex in this case as well, and thus the \emph{Legendre-Hadamard conditions} \cite{agn_neff2000diss}
\begin{align}
\label{eq:LHinequalitiesDefinition}
	D_F^2W(F).[\xi\otimes\eta,\xi\otimes\eta]\geq c^+\norm{\xi}^2\norm{\eta}^2
\end{align} 
are satisfied for all $\xi,\eta\in\R^2$ and some $c^+\geq0$. We calculate the derivatives of $W$ in direction of $H\in\R^{2\times2}$ explicitly to find
\begin{align}
\label{eq:planarEnergyFirstDerivative}
	D_FW(F).[H]=\psi^{\prime}(\K(F))\cdot D_F(\K(F)).[H]=\frac{1}{2}\,\psi^{\prime}(\K(F))\cdot\left[\frac{\iprod{2F-\norm{F}^2F^{-T},\,H}}{\det(F)}\right],
\end{align}
thus
\begin{align}
	&\quad D^2W(F).[H,H]\notag\\
		&=\frac{1}{4}\,\psi^{\prime\prime}(\K(F))\cdot\frac{\left(2\iprod{ F,H}-\norm{F}^2\iprod{ F^{-T},H}\right)^2}{\det(F)^2}\notag\\
		&\quad+\frac{1}{2}\,\psi^{\prime}(\K(F))\cdot\det(F)^{-1}\cdot[-4\iprod{ F^{-T},H}\iprod{ F,H}+2\iprod{ H,H}+\norm{F}^2\iprod{ F^{-T},H}^2
		+\norm{F}^2\iprod{ F^{-T}H^TF^{-T},H}]\label{eq:secondderivationequality}\\
		&\geq \frac{1}{2}\,\psi^{\prime}(\K(F))\cdot\det(F)^{-1}\cdot[-4\iprod{\iprod{F^{-T},H}\. F,H} +2\norm{H}^2+\norm{F}^2\iprod{ F^{-T},H}^2 +\norm{F}^2\iprod{ F^{-T}H^TF^{-T},H}]\notag\\
		&\geq \frac{1}{2}\,\psi^{\prime}(\K(F))\cdot\det(F)^{-1}\cdot\Bigl[-2\.(\.\iprod{ F^{-T},H}^2\.\norm{F}^2\.+\.\norm{H}^2) \label{eq:secondderivationInequality}\\
		&\hspace*{20.58em} +2\norm{H}^2+\norm{F}^2\iprod{ F^{-T},H}^2+\norm{F}^2\iprod{ F^{-T}H^TF^{-T},H}\Bigr]\notag\\
		&=\frac{1}{2}\,\psi^{\prime}(\K(F))\cdot\det(F)^{-1}\norm{F}^2\cdot[-\iprod{ F^{-T},H}^2+\iprod{ F^{-T}H^TF^{-T},H}]\,,\label{eq:secondderivation}
\end{align}
where the inequality in \eqref{eq:secondderivationInequality} follows from applying the Cauchy-Schwartz inequality and Young's inequality, with strict inequality holding if and only if $\iprod{F^{-T},H}^2F\neq H$.

Now, let $H=\xi\otimes\eta$ for $\xi,\eta\in\R^2$. Then \eqref{eq:secondderivation} vanishes due to the mapping $F\mapsto\det F$ being rank-one affine.%
\footnote{%
	Note that $\iprod{D_F\left(\det F\right),H}= \iprod{\Cof F,H}=\det F\cdot\iprod{ F^{-T},H}$ and
	\[
		D_F^2\left(\det F\right).[H,H]=\iprod{\Cof F,H}\iprod{F^{-T},H}+\det F\cdot\iprod{-F^{-T}H^TF^{-T},H}=\det F\cdot[\iprod{F^{-T},H}^2-\iprod{ F^{-T}H^TF^{-T},H}]\,.
	\]
}
Furthermore, since $\rank(F)=2$ and $\rank(H)=1$, we find $\iprod{F^{-T},H}^2F\neq H$ and thus
\begin{align}
	D_F^2W(F).\left[\xi\otimes\eta,\xi\otimes\eta\right]>0\,.
\end{align}
Moreover, the Cauchy stress induced by the elastic energy potential $W$ is given by 
\begin{align*}
	\sigma(F)=D_FW(F)\cdot\Cof(F)^{-1}&=\psi^{\prime}(\K(F))\cdot\left[\frac{FF^T}{\det(F)^2}-\frac{1}{2}\frac{\norm{F}^2\cdot\id}{\det(F)^2}\right]
	=\psi^{\prime}(\K(F))\cdot\left[\frac{FF^T}{\det(F)^2}-\frac{\K(F)}{\det(F)}\cdot\id\right]\,;
\end{align*}
note that $\sigma(F)_{F=\id}=0$. 
Since $\psi$ is assumed to be strictly monotone increasing on $[1,\infty)$, and since $\K(F)=1$ if and only if $F$ is conformal, $W$ attains its global minimum exactly on the set $\CSO(2)$ of conformal matrices. In particular, if $\K(F)=1$, then $FF^T=\det(F)\cdot\id$ and thus 
\begin{align}\label{eq:cauchy-isochoric}
\sigma(F)=\psi^{\prime}(1)\cdot\left[\frac{FF^T}{\det(F)^2}-\frac{1}{\det(F)}\cdot\id\right]=0.
\end{align}

\subsubsection{Approximations in linearized elasticity}

We now take a closer look at the linearization of the energy potential $W$. Applying \eqref{eq:secondderivationequality} to $F=\id$ and $H=\grad u\in\R^{2\times2}$, we find
\begin{align}
	D^2W(\id).[\grad u,\grad u]&=\frac{1}{2}\,\psi^{\prime}(1)\left[-2\iprod{\id,\grad u}^2+2\iprod{\grad u,\grad u}+2\iprod{\grad u^T,\grad u}\right]
	=2\,\psi^{\prime}(1)\,\norm{\dev\sym\grad u}^2\,,
\end{align}
with $\dev_nX=X-\frac{1}{n}\tr(X)\cdot\id_n$, hence the quadratic potential for linearized elasticity corresponding to $W$ is given by
\begin{align}
	W_{\text{lin}}(\grad u)=\mu\,\norm{\dev_2\sym\grad u}^2
\end{align}
with $\mu=\psi^{\prime}(1)$. Moreover,
\begin{align}
	DW_{\text{lin}}(\grad u).[\widetilde{H}]=2\,\mu\,\iprod{\dev\sym\grad u,\widetilde{H}}\,,\qquad	D^2W_{\text{lin}}(\grad u).[\widetilde{H},\widetilde{H}]=2\,\mu\,\norm{\dev_2\sym\widetilde{H}}^2
\end{align} 
for $\widetilde{H}\in\R^{2\times2}$ and thus, for $\xi,\eta\in\R^2$,
\begin{align}
	D^2W_{\text{lin}}(\grad u).[\xi\otimes\eta,\xi\otimes\eta]=2\mu\,\norm{\dev\sym(\xi\otimes\eta)}^2=\mu\,\norm{\xi}^2\norm{\eta}^2\,.
\end{align}
Therefore, $W_{\text{lin}}$ is also strictly elliptic, i.e.\ satisfies \eqref{eq:LHinequalitiesDefinition} with a positive constant $c^+=\mu$. Furthermore, the kernel of the linearized energy is given by 
\begin{align}\label{eq:nullspace}
	W_{\text{lin}}(\grad u)=0&\iff \dev\sym\grad u=0\notag\\
	&\iff u(x)=\frac{1}{2}\left[2\iprod{\matr{-\gamma\\\beta},x}\,x-\matr{-\gamma\\\beta} \norm{x}^2\right]
	+(\widehat{p}\,\id+\widehat{A})\,x+\widehat{b}
\end{align}
with $\widehat{A}\in\so(2)$, $\widehat{b}\in\R^2$ and $\beta$, $\gamma$, $\widehat{p}\in\R$, where $\so(2)$ denotes the set of skew symmetric $2\!\times\!2$--matrices \cite{agn_neff2009subgrid}. This in turn is equivalent to the statement $\sigma_{\text{lin}}(F)=2\mu\,\dev\sym\grad u=0$.

If we consider a quadratic approximation of the conformal mapping $\varphi\colon\Omega\to\R^2$ with $\varphi(x)=\frac{1}{\norm{x}^2}\cdot(x_1,-x_2)^T$, which was discussed in Section \ref{sec:conformalmappings}, we obtain 
\begin{align}\label{linearization}
u(x)=\frac{1}{2}\left[2\.\iprod{
		\matr{
			16\\[.14em]
			0
		}
		,x}\,x
		-\matr{
			16\\[.14em]
			0
		}
		\norm{x}^2\right]-13\cdot\id
		\cdot x
		+\matr{
			6\\[.14em]
			0
		}
		\,.
\end{align}
This expression corresponds to the representation in \eqref{eq:nullspace}, thus $W_{\text{lin}}(\grad u)=0$.
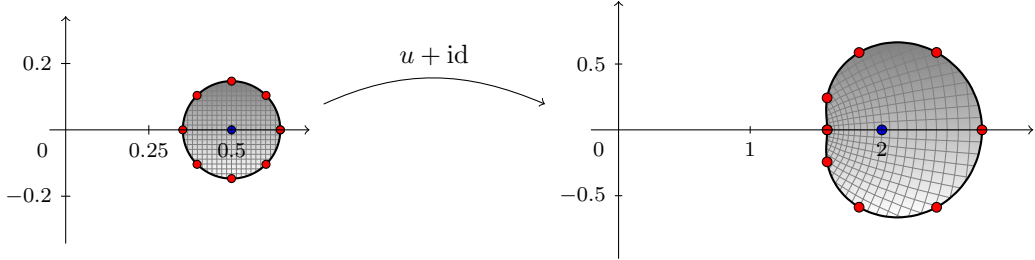
\begin{figure}[H]
	\centering
	\begin{tikzpicture}
	\OmegaSetDefaults
	\OmegaSetGridSize{.0147}{.0147}
	\OmegaSetOutlineSampleCount{42}
	\OmegaSetOutlineStyle{thick, color=black}
	\OmegaSetShadingStyle{color = lightgray}
	\OmegaSetGridStyle{help lines}
	\def\circleCenterX{0.5}
	\def\circleCenterY{0}
	\def\circleRadius{0.147}
	\def\stepSize{45}
	\OmegaSetOutlineToCircle{\circleCenterX}{\circleCenterY}{\circleRadius}
	\OmegaSetSmoothOutline
		\begin{scope}[scale=4.41]
			\OmegaSetGridSampleCount{7}
			\OmegaSetDeformationToId
			\OmegaDraw
			\OmegaSetGridSampleCount{21}
			\foreach \n in {0,\stepSize,...,360}{
				\draw[fill=red]({\circleCenterX+\circleRadius*cos(\n)},{\circleCenterY+\circleRadius*sin(\n)})circle(0.343pt);
			}
			\draw[fill=blue](\circleCenterX,\circleCenterY)circle(0.343pt);
			\draw[->] (0,-0.343) -- (0,0.343);
			\draw[->] (-0.049,0) -- (0.735,0);
			\draw (0.2506,0.014) -- (0.2506,-0.014) node[below] {\footnotesize $0.25$};
			\draw (0.5005,0.014) -- (0.5005,-0.014) node[below] {\footnotesize $0.5$};
			\draw (0.014,0.2002) -- (-0.014,0.2002) node[left] {\footnotesize $0.2$};
			\draw (0.014,-0.2002) -- (-0.014,-0.2002) node[left] {\footnotesize $-0.2$};
			\node[below] at (-0.07,-0.014){\footnotesize $0$};
		\end{scope}
		\begin{scope}[xshift=7.35cm,yshift=0cm,scale=1.75]
			\pgfmathdeclarefunction*{moebiusx}{2}{%
				\pgfmathparse{-8*(#2)^2+8*(#1)^2-12*#1+6}%
			}
			\pgfmathdeclarefunction*{moebiusy}{2}{%
				\pgfmathparse{16*#1*#2-12*#2}%
			}
			\OmegaSetDeformation{moebiusx}{moebiusy}
			\OmegaDraw
			\foreach \n in {0,\stepSize,...,360}{
				\draw[fill=red]({moebiusx(\circleCenterX+\circleRadius*cos(\n),\circleCenterY+\circleRadius*sin(\n))},{moebiusy(\circleCenterX+\circleRadius*cos(\n),\circleCenterY+\circleRadius*sin(\n))})circle(1.05pt);
			}
			\draw[fill=blue]({moebiusx(\circleCenterX,\circleCenterY)},{moebiusy(\circleCenterX,\circleCenterY)})circle(1.05pt);
			\draw[->] (0,-0.98) -- (0,0.98);
			\draw[->] (-0.21,0) -- (3.15,0);
			\draw (1,0.028) -- (1,-0.028) node[below] {\footnotesize $1$};
			\draw (0.028,0.5005) -- (-0.028,0.5005) node[left] {\footnotesize $0.5$};
			\draw (0.028,-0.5005) -- (-0.028,-0.5005) node[left] {\footnotesize $-0.5$};
			\draw (2,0.028) -- (2,-0.028) node[below] {\footnotesize $2$};
			\node[below] at (-0.15,-0.028){\footnotesize $0$};
		\end{scope}
		
		\draw[->] (3.43,0.343) to[out=24.5,in=156.8] (6.37,0.343);
		\draw (4.9,.735) node[above]{$u+\operatorname{id}$};
	\end{tikzpicture}
	\caption{The \enquote{infinitesimal conformal displacement} $u$ given as the the quadratic approximation of the conformal mapping $\varphi\colon\Omega\subset\R^2\setminus\{0\}\to\R^2$ with $\varphi(x)=\frac{1}{\norm{x}^2}\cdot(x_1,-x_2)^T$. Infinitesimal squares are approximately rotated and scaled.}
	\label{fig:infmobius}
\end{figure}
These fundamental connections between conformal mappings, conformally invariant, strictly elliptic energy potentials on $\GLp(2)$ and their linearizations are shown in Fig.\ \ref{fig:commutativediagram}.
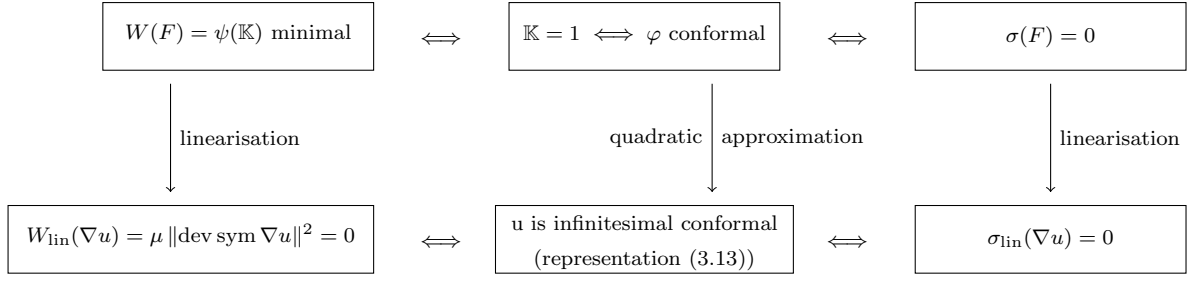
\begin{figure}[H]
	\centering
	\begin{tikzpicture}[scale=1.8]
		\draw (0,0) -- (2,0) -- (2,0.5) -- (0,0.5) -- cycle;
		\draw (3,0) -- (5,0) -- (5,0.5) -- (3,0.5) -- cycle;
		\draw (2.9,-1) -- (5.1,-1) -- (5.1,-1.5) -- (2.9,-1.5) -- cycle;
		\draw (6,0) -- (8,0) -- (8,0.5) -- (6,0.5) -- cycle;
		\draw (6,-1) -- (8,-1) -- (8,-1.5) -- (6,-1.5) -- cycle;
		\draw (-0.7,-1) -- (2,-1) -- (2,-1.5) -- (-0.7,-1.5) -- cycle;
		\draw (2.5,0.125) node[above]{\footnotesize $\iff$};
		\draw (2.5,-1.375) node[above]{\footnotesize $\iff$};
		\draw (5.5,-1.375) node[above]{\footnotesize $\iff$};
		\draw (5.5,0.125) node[above]{\footnotesize $\iff$};
		\draw (7,0.1) node[above]{\footnotesize $\sigma(F)=0$};
		\draw (7,-1.39) node[above]{\footnotesize $\sigma_{\text{lin}}(\grad u)=0$};
		\draw[->] (7,-0.1) -- (7,-0.9);
		\draw (7,-0.5) node[right]{\footnotesize linearisation};
		\draw (1,0.125) node[above]{\footnotesize $W(F)=\psi(\K)$ minimal};
		\draw (4,0.125) node[above]{\footnotesize $\K=1\iff\varphi$ conformal};
		\draw (4,-1.25) node[above]{\footnotesize u is infinitesimal conformal};
		\draw (4,-1.25) node[below]{\footnotesize (representation \eqref{eq:nullspace})};
		\draw[->](0.5,-0.1) -- (0.5,-0.9);
		\draw[->](4.5,-0.1) -- (4.5,-0.9);
		\draw (4.5,-0.5) node[right]{\footnotesize approximation};
		\draw (4.5,-0.5) node[left]{\footnotesize quadratic};
		\draw (0.5,-0.5) node[right]{\footnotesize linearisation};
		\draw (0.65,-1.375) node[above]{\footnotesize $W_{\text{lin}}(\grad u)=\mu\,\norm{\dev\sym\grad u}^2=0$};
	\end{tikzpicture}
	\caption{Connections between conformally invariant, strictly elliptic energies on $\GLp(2)$, their linearisations and conformal mappings.}
	\label{fig:commutativediagram}
\end{figure}

\subsubsection{A general criterion for strict rank-one convexity}
In addition to the sufficient criteria for strict ellipticity of a conformally invariant energy $W\col\GLp(2)\to\R$ given above, i.e.\ the conditions that $\psi^\prime(1)>0$ and $\psi''\geq0$ on $[1,\infty)$ for the representation $\psi$ of $W$ in terms of the distortion $\K$, we can also classify all strictly rank-one convex conformally invariant functions via the representation in terms of the linear distortion $K$, cf.\ Remark \ref{remark:result2DotherEnergies}.

\begin{theorem}
\label{theorem:ellipticityIsochoric}
	Let $W\col\GLp(2)\to\R$ be conformally invariant, i.e.\ $W(\lambda\. Q_1FQ_2)=W(F)$ for all $\lambda\in\R_+$ and all $Q\in\SO(2)$, and let $h\col\Rp\to\R$ and $g\col\Rp\times\Rp\to\R$ denote the uniquely determined functions 
	with
	\[
		W(F)=g(\lambda_1,\lambda_2)=h\left(\frac{\lambda_1}{\lambda_2}\right)=h\left(\frac{\lambda_2}{\lambda_1}\right)
	\]
	for all $F\in\GLp(2)$ with singular values $\lambda_1,\lambda_2$. Then the following are equivalent:
	\begin{itemize}
		\item[i)] $W$ is polyconvex, \qquad\qquad\qquad\qquad\qquad\qquad iv) $h$ is convex on $\Rp$,
		\item[ii)] $W$ is rank-one convex, \qquad\qquad\qquad\qquad\qquad v) $h$ is convex and non-decreasing on $[1,\infty)$.
		\item[iii)] $g$ is separately convex,
	\end{itemize}
\end{theorem}
\begin{proof}
	See \cite[Theorem 3.3]{agn_martin2015rank}.
\end{proof}
\begin{remark}
	Note that $K(F)=\max\{\frac{\lambda_1}{\lambda_2},\frac{\lambda_2}{\lambda_1}\}$ for all $F\in\GLp(2)$ with singular values $\lambda_1,\lambda_2$, thus the representation function $\hhat$ in \eqref{eq:planarConformallyInvariantRepresentations} is simply the restriction of $h$ to $[1,\infty)$.
\end{remark}
\begin{corollary}
\label{corollary:strictEllipticityIsochoric}
	Under the conditions of Theorem \ref{theorem:ellipticityIsochoric}, the following are equivalent:
	\begin{itemize}
	\item[i)] $W$ is strictly rank-one convex, \qquad\qquad\qquad iii) $h$ is strictly convex on $\Rp$,
	\item[ii)] $g$ is separately strictly convex, \qquad\qquad\qquad iv) $h$ is strictly convex and increasing on $[1,\infty)$.
	\end{itemize}
\end{corollary}
\begin{proof}
See Appendix \ref{appendix:strictConvexity}.
\end{proof}

\subsection{A three-dimensional example}
\label{section:threeDimensionalEnergy}

In the three-dimensional case, we consider the conformally invariant energy
\[
	W\colon\GLp(3)\to\R\,,\quad W(F)=\tr\left(\frac{C}{\det(C)^\frac{1}{3}}\right)-3=\frac{\norm{F}^2}{\det(F)^\frac{2}{3}}-3\,,
\]
where $C=F^TF$ denotes the right Cauchy-Green tensor. To check for (strict) rank-one convexity, we compute
\begin{align}
	D_FW(F).[H]&=\frac{\iprod{2F-\frac{2}{3}\norm{F}^2\cdot F^{-T},H}}{\det(F)^\frac{2}{3}},\\
	D^2W(F).[H,H]&=-\frac{8}{3}\frac{\iprod{ F^{-T},H}\iprod{ F,H}}{\det(F)^\frac{2}{3}}+2\frac{\iprod{ H,H}}{\det(F)^\frac{2}{3}}
	+\frac{4}{9}\frac{\norm{F}^2}{\det(F)^\frac{2}{3}}\iprod{ F^{-T},H}^2+\frac{2}{3}\frac{\norm{F}^2}{\det(F)^\frac{2}{3}}\iprod{ F^{-T}H^TF^{-T},H}\notag\\
	&\geq \frac{4}{3}\left(\frac{\iprod{ F^{-T},H}^2\norm{F}^2+\norm{H}^2}{\det(F)^\frac{2}{3}}\right)+2\frac{\norm{H}^2}{\det(F)^\frac{2}{3}}
	+\frac{4}{9}\frac{\norm{F}^2}{\det(F)^\frac{2}{3}}\iprod{ F^{-T},H}^2\notag\\
	&\quad+\frac{2}{3}\frac{\norm{F}^2}{\det(F)^\frac{2}{3}}\iprod{ F^{-T}H^TF^{-T},H}\,.
\end{align}
Hence for $H\colonequals\xi\otimes\eta$ with $\xi,\eta\in\R^3$, we obtain
\begin{align*}
	D^2W(F).[H,H]\geq\frac{4}{3}\cdot\frac{1}{\det(F)^\frac{2}{3}}\cdot\norm{\xi}^2\norm{\eta}^2+2\cdot\frac{1}{\det(F)^\frac{2}{3}}\cdot\norm{\xi}^2\norm{\eta}^2
	=\frac{10}{3}\cdot\det(F)^{-\frac{2}{3}}\cdot\norm{\xi}^2\norm{\eta}^2
	= c\cdot\norm{\xi}^2\norm{\eta}^2
\end{align*}
with $c=\frac{10}{3}\cdot\det(F)^{-\frac{2}{3}}$, thus $W$ is strictly rank-one convex (and even polyconvex \cite{Hartmann_Neff02}).
The Cauchy stress induced by $W$ is given by
\begin{align}
	\sigma(F)&=D_FW(F)\cdot\Cof(F)^{-1}
	=\frac{2F-\frac{2}{3}\norm{F}^2\cdot F^{-T}}{\det(F)^\frac{2}{3}}\cdot\frac{1}{\det(F)}\cdot F^T\notag\\
	&=2\cdot\frac{FF^T}{\det(F)^\frac{2}{3}}\cdot\frac{1}{\det(F)}-\frac{2}{3}\cdot\frac{\norm{F}^2\cdot\id}{\det(F)^\frac{2}{3}}\cdot\frac{1}{\det(F)}\,;
\end{align}
note that $\sigma(F)|_{F=\id}=0$. Moreover, if $F\in\CSO(3)$, then
\begin{align*}
	\sigma(F)
	=\frac{2}{\det(F)}\cdot\id-\frac{2}{3}\cdot\dyniprod{ \frac{FF^T}{\det(F)^\frac{2}{3}},\id}\cdot\frac{1}{\det(F)}\cdot\id
	=\frac{2}{\det(F)}\cdot\id-\frac{2}{\det(F)}\cdot\id=0\,,
\end{align*}
thus the mapping $x\mapsto\sigma(\grad\varphi(x))$ is constant on $\Omega$ if $\varphi$ is a conformal mapping.

We also remark that for arbitrary dimension $n\geq3$, the energy function $W\colon\GLp(n)\to\R$ with $W(F)=\frac{\norm{F}^2}{\det(F)^\frac{2}{n}}-n$ is strictly rank-one convex on $\GLpn$ as well.
\section{Main result}\label{sec:results}
We can now construct our counterexample to the conjecture that homogeneous Cauchy stress implies homogeneity of the deformation for strictly rank-one convex elastic energies. The first result concerns the planar case and a purely isochoric (conformally invariant) energy expression.
\begin{proposition}
\label{prop:generalcase2D}
	For $\K(F)=\frac{1}{2}\frac{\norm{F}^2}{\det F}$ and
	$K(F)=\frac{\opnorm{F}^2}{\det F}=\frac{\lambdamax}{\lambdamin}$, let $W\col\GLp(2)\to\R$ be given by
	\begin{align}
		W(F) =
		\frac{\lambdamax^2}{\lambdamin^2}-1 = K(F)^2-1 = \left( \K(F)+\sqrt{\K(F)^2-1} \right)^2-1 \equalscolon \psi(\K(F))
	\end{align}
	for all $F\in\GLp(2)$ with singular values $\lambdamax\geq\lambdamin$, and let $\varphi\col\Omega\to\R^2$ be any non-affine conformal deformation of a planar body $\Omega\subset\R^2$. Then the Cauchy stress corresponding to the non-homogeneous deformation $\varphi$ induced by the strictly elliptic isotropic energy $W$ is constant.
\end{proposition}
\begin{proof}
	According to Corollary \ref{corollary:strictEllipticityIsochoric}, the given energy is strictly rank-one convex (and indeed polyconvex). Moreover, the deformation is non-affine, i.e.\ non-homogeneous by assumption; cf.\ Section \ref{sec:conformalmappings} for the existence of non-trivial examples of conformal mappings. Finally, since for any isochoric energy the Cauchy stress corresponding to $F=\grad\varphi(x)\in\CSO(2)$ is zero (cf.\ \eqref{eq:cauchy-isochoric}) for all $x\in\Omega$, the induced Cauchy stress is constant.
\end{proof}
\begin{remark}
	In particular, the Cauchy stress response relation is highly non-invertible in this case.
\end{remark}
With the energy considered in Section \ref{section:threeDimensionalEnergy}, the above result immediately applies to the three-dimensional case as well.
\begin{proposition}\label{prop:generalcase3D}
	Let $W\colon\GLp(3)\to\R$ be given by
	\begin{align}\label{eq:generalcase}
		W(F)=\frac{\norm{F}^2}{\det(F)^\frac{2}{3}}-3=\tr\left(\frac{C}{(\det C)^\frac{1}{3}}-\id\right)
	\end{align}
	with $C=F^TF$ (right Cauchy-Green tensor) and let $\widetilde{\varphi}\col\widetilde{\Omega}\to\R^3$ be any non-affine conformal deformation of a body $\widetilde{\Omega}\subset\R^3$. Then the Cauchy stress corresponding to the non-homogeneous deformation $\widetilde{\varphi}$ induced by the strictly elliptic isotropic energy $W$ is constant.
	\pushQED{\qed}\qedhere\popQED
\end{proposition}
Notice that in Proposition \ref{prop:generalcase3D}, the only admissible non-affine conformal mappings are Möbius transformations (see Remark \ref{rem:moebiusconformal}). As an example, consider the mapping
\begin{align}\label{eq:conformalmapping3}
	\widetilde{\varphi}\col\Omega\subset\R^3\setminus\{0\}\to\R^3\,,\quad
	\widetilde{\varphi}(x)=\frac{1}{\norm{x}^2}\left(x_1,-x_2,x_3\right)^T\,.
\end{align}
Then, similar to \eqref{eq:detconformal},
\begin{align}\label{eq:conformal3D}
	\grad\widetilde{\varphi}(x)=\frac{1}{\norm{x}^4}\cdot\left(\begin{array}{ccc}
										\norm{x}^2-2x_1^2 & -2x_1x_2 & -2x_1x_3\\
										2x_1x_2 & -\norm{x}^2+2x_2^2 & 2x_2x_3\\
										-2x_1x_3 & -2x_2x_3 & \norm{x}^2-2x_3^2
										\end{array}\right), \quad
	\det(\grad\widetilde{\varphi})=\frac{1}{\norm{x}^6}.
\end{align}

\subsection{A physically viable example}
In our examples, we observe that the reference configuration is stress free, i.e.\ $\sigma(\id)=0$. However, since $\sigma(F)=0$ for all $F\in\CSOn$, the stress free configuration is not unique. This constitutive shortcoming can be circumvented by adding a volumetric energy term of the form $F\mapsto f(\det F)$ for some differentiable function $f\col(0,\infty)\to\R$ such that $f'(t)\neq0$ for all $t\neq1$.

In order for such an additively coupled energy to be suitable for our purpose, of course, it must be ensured that a constant Cauchy stress tensor field can still be achieved with an inhomogeneous conformal deformation. We therefore choose $f$ such that $f'$ is constant, i.e.\ such that $f$ is linear, on an interval $[a,b]\subset(1,\infty)$. For example, consider the function
\begin{equation}\label{eq:volumetricEnergyLinearOnInterval}
	f\col(0,\infty)\to\R\,,\quad
	f(t)=
	\begin{cases}
		\ln^2(t) &: t<e\\
		1+\frac{2(t-e)}{e} &: e\leq t\leq c\\
		1+\frac2e\.(e^{t-c}+(c-e-1)) &: c\leq t
	\end{cases}
\end{equation}
for some $c>e$, as shown in Fig.\ \ref{fig:Wvol}. Then $f$ is convex and continuously differentiable with $f'(t)=\frac2e$ for all $t\in[e,c]$. In particular, adding the volumetric term $F\mapsto f(\det F)$ to an energy function preserves (strict) rank-one convexity, and its contribution to the Cauchy stress tensor \cite{richter1948} is constant if $\det(F)\in[e,c]$. Combined with Propositions \ref{prop:generalcase2D} and \ref{prop:generalcase3D}, these observations immediately imply the following main result.

\begin{proposition}\label{prop:final}
For $n=2$ and $n=3$, let $W\col\GLpn\to\R$ be given by
\begin{align}
\label{eq:finalexample2D}
	W(F)
	&= \Wiso\left(\frac{F}{(\det F)^\frac{1}{2}}\right)+\Wvol(\det F)
	= K(F)^2 + f(\det F) - 1 = \frac{\lambdamax^2}{\lambdamin^2} + f(\lambdamax\.\lambdamin) - 1
\intertext{and}
\label{eq:finalexample3D}
	W(F)
	&= \Wiso\left(\frac{F}{(\det F)^\frac{1}{3}}\right)+\Wvol(\det F)
	=\left(\frac{\norm{F}^2}{\det(F)^\frac{2}{3}}-3\right)+f(\det F)\,,
\end{align}
respectively, where $f$ is given by \eqref{eq:volumetricEnergyLinearOnInterval}. Then for any non-affine conformal deformation mapping $\varphi\col\Omega\to\R^n$ on an open and bounded domain $\Omega\subset\R^n$ such that $\det(\grad\varphi(x))\in[e,c]$ for all $x\in\Omega$, the Cauchy stress tensor field induced by $W$ corresponding to the non-homogeneous deformation $\varphi$ is constant.
\end{proposition}

\begin{figure}[H]
	\centering
	\begin{minipage}[c]{0.45\linewidth}
	\centering
	\begin{tikzpicture}
	\begin{axis}[axis x line=center,
		axis y line=center,
		xtick={0,1,2.73,4.9},
		ytick={0},
		xticklabels={0,1\vphantom{b},$e\vphantom{b}$,$c\vphantom{b}$},
		yticklabels={0},
		xlabel={$t$},
		ylabel={$f(t)$},
		xlabel style={below right},
		ylabel style={above left},
		xmin=-0.14,
		xmax=7.7,
		ymin=-1,
		ymax=6.3,
		samples=42,
		width=.931\linewidth,
		height=.6\linewidth,
		domain=0:4.2
	]
		\addplot[domain=0.14:e,color=blue] { ln(x)^2 };
		\addplot[domain=e:4.9,color=red] { 1+2/e*(x-e) } node[pos=.21,below right] {linear};
		\addplot[domain=4.9:6.37,color=blue] { 1+2/e*(exp(x-4.9)+(4.9-e-1))};
		\draw[dotted] (axis cs: 1,0) -- (axis cs: 1,3);
		\draw[dotted] (axis cs: e,0) -- (axis cs: e,3);
		\draw[dotted] (axis cs: 4.9,0) -- (axis cs: 4.9,3);
	\end{axis}
	\end{tikzpicture}
	\end{minipage}
	\begin{minipage}[c]{0.45\linewidth}
	\centering
	\begin{tikzpicture}
	\begin{axis}[
		axis x line=center,
		axis y line=center,
		xtick={0,1,2.73,4.9},
		ytick={0,0.7329},
		xticklabels={0,1\vphantom{b},$e\vphantom{b}$,$c\vphantom{b}$},
		yticklabels={0,$\frac2e$},
		xlabel={$t$},
		ylabel={$f^\prime(t)$},
		xlabel style={below right},
		ylabel style={above left},
		xmin=0.245,
		xmax=7,
		ymin=-2.45,
		ymax=3.43,
		samples=42, 
		width=.931\linewidth,
		height=.6\linewidth,
		domain=0:4.2
	]
		\addplot[domain=0.343:e,color=blue] { 2*ln(x)/x)};
		\addplot[domain=e:4.9,color=red] { 2/e)};
		\addplot[domain=4.9:7,color=blue] { 2/e*(exp(x-4.9) };
		\draw[dotted] (axis cs: 0,0.7329) -- (axis cs: 2.45,0.7329);
	\end{axis}	
	\end{tikzpicture}
	\end{minipage}
	\caption{Visualization of the volumetric energy term. Note that $f$ is convex but not strictly convex.}
	\label{fig:Wvol}
\end{figure}
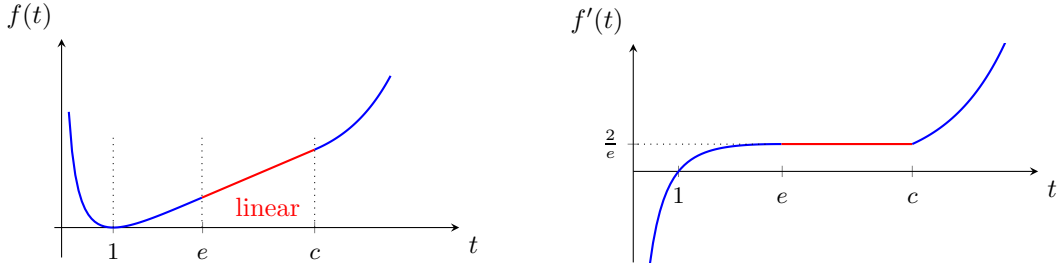
\begin{proof}
	Due to the rank-one convexity of $W_{\text{vol}}$, which follows via polyconvexity from the convexity of $f$, the energy $W$ is again strictly rank-one convex. The Cauchy stress for the non-affine conformal mapping $\widetilde{\varphi}$ in \eqref{eq:conformalmapping3} is given by
	\begin{align}
		\sigma(F)&=\sigma_{W_{\text{iso}}}(F)+\sigma_{W_{\text{vol}}}(F)=0+D_F\left(W_{\text{vol}}(\det F)\right)\cdot\Cof(F)^{-1}\notag\\
		&=\left[W^\prime_{\text{vol}}(\det F)\cdot \Cof(F)\right]\cdot\Cof(F)^{-1}=f^\prime(\det F)\cdot\id
		=\frac2e\cdot\id\,,
	\end{align}
	i.e.\ the Cauchy stress is constant, if $\det(\grad\varphi(x))\in(a,b)$ for all $x\in\Omega$.
\end{proof}
\begin{remark}
	Of course, a conformal mapping $\varphi$ with the required properties can easily be obtained via a simple reflection at a sphere if the domain $\Omega$ is chosen accordingly (as a subset of the annulus on which the condition on the determinant is satisfied).
\end{remark}
\begin{remark}
\label{remark:constitutivePropertiesOfExample}
	Observe carefully that the elastic energy $W$ in our final example, given by \eqref{eq:finalexample2D} and \eqref{eq:finalexample3D}, is polyconvex with $W(F)\to+\infty$ as $\det F\to0$, $W(F)\to+\infty$ for $\det F\to+\infty$ and $W(F)\to+\infty$ for $\norm{F}\to\infty$. Furthermore, the linearization of $W$ is well posed with $W_{\text{lin}}(\grad u)=2\,\norm{\dev_3\sym\grad u}^2+\frac{f^{\prime\prime}(1)}{2}\left(\tr(\grad u)\right)^2$, and $W$ has a unique stress-free state. It is also easy to see that the volumetric term $f$ can be modified such that the resulting energy $W$ is $C^\infty$-regular (and thus strictly Legendre-Hadamard elliptic as well).
\end{remark}
\begin{remark}
	While our counterexample in Proposition \ref{prop:final} is based on a special form of the volumetric-isochoric split, we believe that the situation is generic, i.e.\ that possible counterexamples are not restricted to this particular class of volumetric-isochoric split energy functions.
\end{remark}
\begin{remark}
	As an immediate consequence, Proposition \ref{prop:final} shows that a strictly rank-one convex energy may give rise to a non-invertible Cauchy stress-stretch relation. On the other hand, it has previously been shown \cite{agn_neff2015exponentiatedI} that an invertible Cauchy stress-stretch relation based on a volumetric-isochoric split may be locally non-elliptic. Therefore, invertibility of the mapping $B\to\sigma(B)$ from the left Cauchy-Green tensor $B=FF^T$ to the Cauchy stress is in general not related to rank-one convexity. Note that in linear isotropic elasticity, invertibility of the stress response $\varepsilon\to\sigma_\text{lin}(\varepsilon)$ to the infinitesimal strain tensor $\varepsilon=\sym\grad u$ amounts to strict convexity of the energy and implies (but is not implied by) rank-one convexity.
\end{remark}
\begin{remark}
\label{remark:result2DotherEnergies}
	According to Corollary \ref{corollary:strictEllipticityIsochoric}, Proposition \ref{prop:final} still holds if the isochoric part $K(F)^2-1$ in \eqref{eq:finalexample2D} is replaced by any strictly convex, increasing function of $K(F)$.
\end{remark}

\subsection{Connection to previous results}
As indicated in Section \ref{section:introduction}, for a strictly rank-one convex energy, a homogeneous Cauchy stress tensor field cannot correspond to a non-homogeneous deformation if the deformation gradient has discrete values, i.e.\ if the deformation is piecewise affine linear and satisfies the Hadamard jump condition, cf.\ Fig.\ \ref{fig:laminate}. Here, however, we have obtained a homogeneous (constant) Cauchy stress tensor which corresponds to a non-homogeneous deformation (a non-affine conformal mapping) for a strictly rank-one convex energy potential. Nevertheless, this result does not contradict the statement made in \cite{agn_neff2016injectivity}, since conformal mappings are not compatible with the Hadamard jump condition.

Recall that in order to satisfy the jump condition, the deformation gradients along the interface must be rank-one connected \cite{agn_mihai2016hyperelastic}. However, it can be shown \cite[Lemma 8.25]{rindler2018calculus} that $\rank(F_1-F_2)\neq1$ for all $F_1,F_2\in\CSO(2)$; note that
each $F\in\CSO(2)$ is of the form
\begin{align}
	F=
	\matr{
		a & b\\
		-b & a
	}
	\quad\text{with}\quad
	a,b\in\R
\end{align}
and thus for any $F_1,F_2\in\CSO(2)$ with $F_1\neq F_2$,
\[
	\det(F_1-F_2)
	= \det\matr{a_1-a_2 & b_1-b_2\\ -b_1+b_2 & a_1-a_2}
	= (a_1-a_2)^2 + (b_1-b_2)^2
	> 0\,,
\]
hence $\rank(F_1-F_2)=2$.
In particular, for a (planar) conformal mapping $\varphi\col\Omega\subset\R^2\to\R^2$, the deformation gradients $\grad\varphi(x_1),\grad\varphi(x_2)$ at $x_1,x_2\in\Omega$ cannot differ by a rank-one matrix.
\begin{figure}[H]
		\centering
		\begin{tikzpicture}[scale=2]
				\def\shiftone{1.4}
			\def\shifttwo{-0.49}
			\def\numberOfSteps{5}
			\def\numberOfInnerLines{5}
			\def\height{(1/(2*\numberOfSteps))}
			\foreach \n in {1,...,\numberOfSteps}{
				\def\i{(\n-1)}
				\begin{scope}[cm={1,0,\shiftone,1,({\i*\height*(\shiftone+\shifttwo)},{2*\i*\height})}]
					\draw[blue] (0,0) -- (0,{\height}) -- (1,{\height}) -- (1,0) -- cycle;
					\foreach \k in {1,...,\numberOfInnerLines}{
						\draw[blue,very thin] ({\k/(\numberOfInnerLines+1)},0) -- ({\k/(\numberOfInnerLines+1)},{\height});
					}
				\end{scope}
				\begin{scope}[cm={1,0,\shifttwo,1,({\height*(\i*(\shiftone+\shifttwo)+\shiftone)},{\height*(2*\i+1)})}]
					\draw[red] (0,0) -- (0,{\height}) -- (1,{\height}) -- (1,0) -- cycle;
					\foreach \k in {1,...,\numberOfInnerLines}{
						\draw[red,very thin] ({\k/(\numberOfInnerLines+1)},0) -- ({\k/(\numberOfInnerLines+1)},{\height});
					}
				\end{scope}
			}
			
			\draw[dashed,thick] ({\height*(\shiftone+\shifttwo)/2},0) -- ({\height*(\shiftone+\shifttwo)/2+(\shiftone+\shifttwo)/2},1) -- ({\height*(\shiftone+\shifttwo)/2+1+(\shiftone+\shifttwo)/2},1) -- ({\height*(\shiftone+\shifttwo)/2+1},0) -- cycle;
	\end{tikzpicture}
	\caption{Structure of a rank-one laminate \cite{agn_neff2018cauchystress_proceedings,agn_schweickert2018nonhomogeneous}.}
	\label{fig:laminate}
\end{figure}
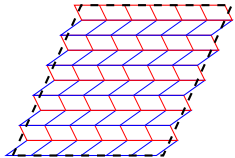
\section{Conclusion}
We have constructed non-trivial, strictly rank-one convex (polyconvex) elastic energies which induce a constant Cauchy stress tensor field for certain non-homogeneous deformations.
Our result motivates a further investigation of the role of invertibility assumptions on the Cauchy stress response function in terms of the left Cauchy-Green tensor $B$. Indeed, a constant Cauchy stress for non-constant stretches being a rather strange phenomenon in idealized nonlinear elasiticity, we are not yet in a position to judge on the appropriateness of such constitutive assumptions in nonlinear hyperelasticity.
\section*{Acknowledgement}
The last author is indepted to Prof. Konstantin Volokh (Technion, Haifa) for interesting discussions concerning the relevance of our counterexample.
\begin{appendix}
\begin{footnotesize}
\section{Necessary and sufficient conditions for rank-one convexity}
\label{appendix:knowlesSternberg}

\renewcommand{\dd}{\displaystyle}

\noindent
The following ellipticity condition for the planar case is due to Knowles and Sternberg {\footnotesize \cite[p.~9]{knowles1976failure} (cf.\ \cite{knowles1978failure,silhavy1997mechanics})}.

\begin{lemma}[{Knowles and Sternberg \cite{knowles1976failure,knowles1978failure}, cf.\ \cite[p.~308]{silhavy1997mechanics}}]
\label{lemma:knowlesSternberg}
	Let $W\col\GLp(2)\to\R$ be an objective and isotropic function with $W(F)=g(\lambda_1,\lambda_2)$ for all $F\in\GLp(2)$ with singular values $\lambda_1,\lambda_2$, where $g\col\Rp^2\to\R$ is two-times continuously differentiable. Then $W$ is Legendre-Hadamard elliptic on $\GLp(2)$ if and only if $g$ satisfies the following conditions for all $(\lambda_1,\lambda_2)\in\Rp^2$:
	\begin{alignat*}{2}
		\text{i)}&\qquad g_{11}\geq 0 \qquad\text{and}\qquad g_{22}\geq0\,,\\
		\text{ii)}&\qquad \frac{\lambda_1\.g_1-\lambda_2\.g_2}{\lambda_1-\lambda_2}\geq 0
			&&\qquad\text{if }\; \lambda_1\neq \lambda_2\,, \hspace*{.21\textwidth}\\
		\text{iii)}&\qquad g_{11}-g_{12}+\frac{g_1}{\lambda_1}\geq 0 \quad\text{ and }\quad g_{22}-g_{12}+\frac{g_2}{\lambda_2}\geq 0
			&&\qquad\text{if }\; \lambda_1=\lambda_2\,,\\
		\text{iv)}&\qquad \sqrt{g_{11}\,g_{22}}+g_{12}+\frac{g_1-g_2}{\lambda_1-\lambda_2}\geq 0
			&&\qquad\text{if }\; \lambda_1\neq \lambda_2\,,\\
		\text{v)}&\qquad \sqrt{g_{11}\,g_{22}}-g_{12}+\frac{g_1+g_2}{\lambda_1+\lambda_2}\geq 0\,,
	\end{alignat*}
	where $g_i=\pdd{g}{\lambda_i}(\lambda_1,\lambda_2)$ and $g_{ij}=\pdd[2]{g}{\lambda_1\,\partial\lambda_2}(\lambda_1,\lambda_2)$.
	Furthermore, if all the above inequalities are strict, then $W$ is strictly elliptic.
\end{lemma}

\section{Some observations on strict convexity}
\label{appendix:strictConvexity}
In order to prove Corollary \ref{corollary:strictEllipticityIsochoric}, we will discuss some criteria for the \emph{strict} convexity of real-valued functions.
In the following, we will assume that $C\subset V$ is a convex subset of a vector space $V$ and $I\subset\R$ is an interval.

\begin{definition}
\label{definition:semiStrictConvexity}
	A function $f\col C\to\R$ is called \emph{semi-strictly convex} if
	\begin{itemize}
		\item[i)] $f$ is convex and
		\item[ii)] if $f(a\.x+(1-a)y)=a\.f(x)+(1-a)\.f(y)$ for some $a\in(0,1)$, then $f(x)\neq f(y)$.
	\end{itemize}
\end{definition}

\begin{remark}
	A convex function is semi-strict convexity if and only if it is not constant along straight lines. Note that strict convexity can similarly be expressed as the non-\emph{linearity} along straight lines. It is easy to see that an analogous concept can be applied to rank-one convexity, although a separate definition seems to be redundant (instead, the definition can simply be applied to the mapping's restriction to rank-one lines).
\end{remark}

\begin{lemma}
\label{lemma:semstrictConvexityImpliedByStrictComposition}
	If $f\col C\to R\subset\R$ is convex and there exists $g\col R\to\R$ such that $g\circ f$ is strictly convex, then $f$ is semi-strictly convex.
\end{lemma}
\begin{proof}
	Let $a\in(0,1)$ and $x\neq y$ for $x,y\in C$. If $f(a\.x+(1-a)y)=a\.f(x)+(1-a)\.f(y)$ and $f(x)=f(y)$, then
	\[
		g(f(a\.x+(1-a)y)) = g(a\.f(x)+(1-a)\.f(y)) = g(f(x)) = a\.g(f(x))+(1-a)\.g(f(y))
	\]
	in contradiction to the strict convexity of $g$.
\end{proof}

\begin{lemma}
\label{lemma:semstrictConvexityImpliesStrictCompositions}
	If $f\col C\to R\subset\R$ is semi-strictly convex and $h\col R\to\R$ is strictly convex and (strictly) monotone, then $h\circ f$ is strictly convex.
\end{lemma}
\begin{proof}
	Let $a\in(0,1)$ and $x\neq y$ for $x,y\in C$.
	
	If $f(x)\neq f(y)$, then
	\[
		h(f(a\.x+(1-a)y)) \leq h(a\.f(x)+(1-a)\.f(y)) < a\.h(f(x))+(1-a)\.h(f(y))
	\]
	due to the strict monotonicity the convexity of $f$ and the monotonicity and strict convexity of $h$.
	
	If $f(x)=f(y)$, then $f(a\.x+(1-a)y)< a\.f(x)+(1-a)\.f(y)$ due to the semi-strict convexity of $f$ and thus
	\[
		h(f(a\.x+(1-a)y)) < h(a\.f(x)+(1-a)\.f(y)) = h(f(x)) = a\.h(f(x))+(1-a)\.h(f(y))
	\]
	due to the strict convexity of $h$.
\end{proof}

\begin{lemma}
	If $f\col I\to\R$ is analytic and convex on an interval $I\subset\R$, then $f$ is either strictly convex or linear.
\end{lemma}
\begin{proof}
	If $f$ is analytic with $f''\geq0$ on the interval $I$, then either $f''\equiv0$ or $f''>0$ on the interior of $I$.
\end{proof}

\begin{corollary}
	If $f\col \R^n\to\R$ is analytic and convex, then $f$ is either strictly convex or the \emph{restriction} of $f$ \emph{to some line} in $\R^n$ is linear.
\end{corollary}

\begin{remark}
	Observe that the (analytic) mapping $F\mapsto\frac{\lambdamax}{\lambdamin}$ is linear on the \enquote{rank-one-(half-)line} $\id+\R\cdot e_1\otimes e_1$, but not constant on this line (which would contradict the semi-strict convexity).
\end{remark}

The above observations allow us to prove Corollary \ref{corollary:strictEllipticityIsochoric} by explicitly demonstrating the strictness of a single composition with the mapping $F\mapsto\frac{\lambdamax}{\lambdamin}$.

	\begin{corollary*}[Corollary \ref{corollary:strictEllipticityIsochoric}]
		Let $W\col\GLp(2)\to\R$ be conformally invariant, and let $h\col\Rp\to\R$ and $g\col\Rp\times\Rp\to\R$ denote the uniquely determined functions with
		$W(F)=g(\lambda_1,\lambda_2)=h\bigl(\frac{\lambda_1}{\lambda_2}\bigr)$
		for all $F\in\GLp(2)$ with singular values $\lambda_1,\lambda_2$. Then the following are equivalent:
		\begin{itemize}
		\item[i)] $W$ is strictly rank-one convex, \qquad\qquad\qquad iii) $h$ is strictly convex on $\Rp$,
		\item[ii)] $g$ is separately strictly convex, \qquad\qquad\qquad iv) $h$ is strictly convex and increasing on $[1,\infty)$.
		\end{itemize}
	\end{corollary*}
%
\begin{proof}
i) $\implies$ ii): Assume $W$ to be strictly rank-one convex. Then
\begin{align}
	g(1+t,a)
	&=W \matr{
			1+t & 0\\
			0 & a		
		}
		= W \left(
		\matr{
			1 & 0\\
			0 & a			
		} + t \cdot e_1 \otimes e_1 \right)\notag\\ 
		&< (1-t) \cdot W
		\matr{
			1 & 0\\
			0 & a
		} +t \cdot W \left(
		\matr{
			1 & 0\\
			0 & a
		} + e_1 \otimes e_1 \right)\notag
		\;=\; (1-t) \cdot g(1,a) + t \cdot g(2,a).						
\end{align}
Thus the function $ s \mapsto g(s,a)$ is strictly convex which, due to the symmetry of $g$, implies that $g$ is separately convex.\\
ii) $\implies$ iii): If $g$ is separately strictly convex, then for $\lambda_1 =1$, $\lambda_2=(1-t) \cdot x + t \cdot y$ and $t \in [0,1]$, we find
\begin{align*}
	h((1-t) \cdot x + t \cdot y) = g((1-t) \cdot x + t \cdot y,1) < (1-t) \cdot g(x,1) + t \cdot g(y,1) = (1-t) \cdot h(x) + t \cdot h(y)\,,
\end{align*}
thus $h$ is strictly convex.\\
iii) $\implies$ iv): It suffices to note that the convexity of $h$ on $\Rp$ implies that $h$ is increasing on $[1,\infty)$ according to Theorem \ref{theorem:ellipticityIsochoric}.
iv) $\implies$ i): First, we show that the energy induced by
$\widetilde{g}(\lambda_1, \lambda_2) = \bigl( \frac{\max\{\lambda_1, \lambda_2\}}{\min\{ \lambda_1, \lambda_2 \}} - 1 \bigr)^2$
is strictly rank-one convex. Assume without loss of generality that $\lambda_1\geq\lambda_2$. Then
$\widetilde{g}(\lambda_1,\lambda_2) = \bigl( \frac{\lambda_1}{\lambda_2} - 1 \bigr)^2$,
and after some computation, we find
\begin{align*}
	\gtilde_{11}
		\;=\; \frac{2}{\lambda_2^2}
		\;>\; \,0
	\qquad&\,\text{and}\,\qquad
	\gtilde_{22}
		\;=\; \frac{2\.\lambda_1\.(3\.\lambda_1-2\.\lambda_2)}{\lambda_2^4}
		\;>\; 0\,,
	\\
	\frac{\lambda_1\.\gtilde_1-\lambda_2\.\gtilde_2}{\lambda_1-\lambda_2}
		&\;=\; \frac{4\.\lambda_1}{\lambda_2^2}
		\;>\; 0\,,
	\\
	\gtilde_{11}-\gtilde_{12}+\frac{\gtilde_1}{\lambda_1}
		&\;=\; \frac{2\.(\lambda_1+\lambda_2)\.(2\.\lambda_1-\lambda_2)}{\lambda_1\.\lambda_2^3}
		\;>\; 0\,,
	\\
	\gtilde_{22}-\gtilde_{12}+\frac{\gtilde_2}{\lambda_2}
		&\;=\; \frac{2\.(\lambda_1+\lambda_2)\.(2\.\lambda_1-\lambda_2)}{\lambda_2^4}
		\;>\; 0\,,
	\\
	\sqrt{\gtilde_{11}\,\gtilde_{22}}+\gtilde_{12}+\frac{\gtilde_1-\gtilde_2}{\lambda_1-\lambda_2}
		&\;=\; 2\,\frac{2\.\lambda_2-\lambda_1+\sqrt{3\.\lambda_1^2-2\.\lambda_1\.\lambda_2}}{\lambda_2^3}
		\;>\; 0\,,
	\\
	\sqrt{\gtilde_{11}\,\gtilde_{22}}-\gtilde_{12}+\frac{\gtilde_1+\gtilde_2}{\lambda_1+\lambda_2}
		&\;=\; 2\,\frac{3\.\lambda_1\.\lambda_2-2\.\lambda_2^2+\lambda_1^2+\sqrt{3\.\lambda_1^4-2\.\lambda_1^3\.\lambda_2}+\lambda_2\.\sqrt{3\.\lambda_1^2-2\.\lambda_1\.\lambda_2}}{\lambda_1\.\lambda_2^3+\lambda_2^4}
		\;>\; 0\,.
\end{align*}
According to the Knowles-Sternberg criterion (Lemma \ref{lemma:knowlesSternberg}), combined with additional calculations by Silhavy \cite{silhavy1997mechanics}, the function $W$ induced by $\gtilde$ is therefore strictly rank-one convex.

Now, for $F\in\GLp(2)$ and $\xi,\eta\in\R^2$, consider the function $f\colon[0,1]\to\R$ with $f(t)\colonequals K(F+t(\xi\otimes\eta))$, where $K(F)=\frac{\lambdamax}{\lambdamin}$. Then $f$ is convex due to Theorem \ref{theorem:ellipticityIsochoric} v) with $h\colonequals\operatorname{id}$. The above computations show that for $g\colon\R\to\R$ with
$g(s)=(s-1)^2$,
the composition $g\circ f$ is strictly convex. Therefore, according to Lemma \ref{lemma:semstrictConvexityImpliedByStrictComposition}, the function $f$ is semi-strictly convex (cf.\ Definition \ref{definition:semiStrictConvexity}).

Finally, under the assumption that $h$ is strictly convex and increasing on $[1,\infty)$, Lemma \ref{lemma:semstrictConvexityImpliesStrictCompositions} states that the mapping
\[
	h\circ f\col[0,1]\to\R\,,\\
	t\mapsto h(f(t)) = h(K(F+t(\xi\otimes\eta)) = W(F+t(\xi\otimes\eta))
\]
is strictly convex for arbitrary $F\in\GLp(2)$ and $\xi,\eta\in\R^2$, which directly implies the strict rank-one convexity of $W$.
\end{proof}

\end{footnotesize}
\end{appendix}

{\footnotesize \printbibliography}

\end{document}